\documentclass{amsart}
\usepackage{amsmath, amssymb, mathrsfs, verbatim, multirow}
\usepackage[all]{xy}

\newtheorem{Teo}{Theorem}[section]
\newtheorem{Prop}[Teo]{Proposition}
\newtheorem{Lema}[Teo]{Lemma}
\newtheorem{Cor}[Teo]{Corollary}

\theoremstyle{definition}
\newtheorem{Def}[Teo]{Definition}

\newtheorem{Obs}[Teo]{Remark}

\newtheorem{Exa}[Teo]{Example}

\newcommand{\R}{\mathbb{R}}

\newcommand{\N}{\mathbb{N}}

\newcommand{\Llr}{\Longleftrightarrow}

\newcommand{\Ch}{\mbox{\rm char}}

\begin{document}
\title{Key polynomials and pseudo-convergent sequences}
\author{Josnei Novacoski}
\author{Mark Spivakovsky}

\begin{abstract}
In this paper we introduce a new concept of key polynomials for a given valuation $\nu$ on $K[x]$. We prove that such polynomials have many of the expected properties of key polynomials as those defined by MacLane and Vaqui\'e, for instance, that they are irreducible and that the truncation of $\nu$ associated to each key polynomial is a valuation. Moreover, we prove that every valuation
$\nu$ on $K[x]$ admits a sequence of key polynomials that completely determines $\nu$ (in the sense which we make precise in the paper). We also establish the relation between these key polynomials and pseudo-convergent sequences defined by Kaplansky.
\end{abstract}

\maketitle
\section{Introduction}
Given a valuation $\nu$ of a field $K$, it is important to understand what are the possible extensions of $\nu$ to $K[x]$. Many different theories have been developed in order to understand such extensions. For instance, in \cite{Mac}, MacLane develops the theory of key polynomials. He proves that given a discrete valuation $\nu$ of $K$, every extention of $\nu$ to $K[x]$ is uniquely determined by a sequence (with order type at most $\omega$) of key polynomials. Recently, M. Vaqui\'e developed a more general theory of key polynomials (see \cite{Vaq}), which extends the results of MacLane for a general valued field (that is, the given valuation of $K$ is no longer assumed to be discrete). At the same time, F.H. Herrera Govantes, W. Mahboub, M.A. Olalla Acosta and M. Spivakovsky developed another definition of key polynomials (see \cite{HOS}). This definition is an adaptation of the concept of generating sequences introduced by Spivakovsky in \cite{Spi1}. A comparison between this two definitions of key polynomials is presented in \cite{Mahboud}.

Roughly speaking, for a given valuation $\mu$ of $K[x]$, a MacLane -- Vaqui\'e key polynomial $\phi\in K[x]$ for $\mu$ is a polynomial that allows us to obtain a new valuation $\mu_1$ of $K[x]$ with $\mu_1(\phi)=\gamma_1>\mu(\phi)$ and $\mu(p)=\mu_1(p)$ for every $p\in K[x]$ with $\deg(p)<\deg(\phi)$ (in this case we denote $\mu_1$ by $[\mu;\mu_1(\phi)=\gamma_1]$). Then, for any valuation
$\nu$ of $K[x]$ one tries to obtain a sequence of valuations $\mu_0,\mu_1,\ldots,\mu_n,\ldots$ with $\mu_0$ a monomial valuation and $\mu_{i+1}=[\mu_i;\mu_{i+1}(\phi_{i+1})=\gamma_{i+1}]$ for a key polynomial $\phi_{i+1}$ for $\mu_i$, such that
\begin{equation}
\nu=\lim \mu_i\label{eq:nu=lim}
\end{equation}
(in the sense that will be defined precisely below). This process does not work in general, that is, the equality (\ref{eq:nu=lim}) may not hold even after one constructs an infinite sequence $\{\mu_i\}$.  This leads one to introduce the concept of ``limit key polynomial".
It is known that valuations which admit limit key polynomials are more difficult to handle. For instance, it was proved by J.-C. San Saturnino (see Theorem 6.5 of \cite{JCSS}), that if a valuation $\nu$ is centered on a noetherian local domain and $\nu$ does not admit limit key polynomials (on any sub-extension $R'\subseteq R'[x]\subseteq R$ with $\dim\ R'=\dim\ R-1$), then it has the local uniformization property (where we assume, inductively, that local uniformization holds for $R'$).

In this paper, we introduce a new concept of key polynomials. Let $K$ be a field and $\nu$ a valuation on $K[x]$. Let $\Gamma$ denote the value group of $K$ and $\Gamma'$ the value group of $K[x]$. For a positive integer $b$, let
$\partial_b:=\frac{1}{b!}\frac{\partial^b}{\partial x^b}$ (this differential operator of order $b$ is sometiems called \textbf{the $b$-th formal derivative}). For a polynomial $f\in K[x]$ let
\[
\epsilon(f)=\max_{b\in \N}\left\{\frac{\nu(f)-\nu(\partial_bf)}{b}\right\}.
\]
A monic polynomial $Q\in K[x]$ is said to be a \textbf{key polynomial} (of level $\epsilon (Q)$) if for every $f\in K[x]$ if $\epsilon(f)\geq \epsilon(Q)$, then $\deg(f)\geq\deg(Q)$.

This new definition offers many advantages. For instance, it gives a criterion to determine, for a given valuation $\nu$ of $K[x]$, whether any given polynomial is a key polynomial for $\nu$. This has a different meaning than in the approach of MacLane-Vaqui\'e. In their approach, a key polynomial allows us to ``extend the given valuation" and here a key polynomial allows us to ``truncate the given valuation". For instance, our definition of key polynomials treats the limit key polynomials on the same footing as the non-limit ones. Moreover, we present a characterization of key polynomial (Theorem \ref{definofkeypol}) which allows us to determine whether a given key polynomial is a limit key polynomial. A more precise comparison between the concept of key polynomial introduced here and that of MacLane -- Vaqui\'e will be explored in a forthcoming paper by Decoup, Mahboub and Spivakovsky.

Given two polynomials $f,q\in K[x]$ with $q$ monic, we call the \textbf{$q$-standard expansion of $f$} the expression
\[
f(x)=f_0(x)+f_1(x)q(x)+\ldots+f_n(x)q^n(x)
\]
where for each $i$, $0\leq i\leq n$, $f_i=0$ or $\deg(f_i)<\deg(q)$. For a polynomial $q(x)\in K[x]$, the \textbf{$q$-truncation of $\nu$} is defined as
\[
\nu_q(f):=\min_{0\leq i\leq n}\{\nu(f_iq^i)\}
\]
where $f=f_0+f_1q+\ldots+f_nq^n$ is the $q$-standard expansion of $f$. In Section 2, we present an example that shows that $\nu_q$ does need to be a valuation. We also prove (Theorem \ref{proptruncakeypolval}) that if $Q$ is a key polynomial, then $\nu_Q$ is a valuation. A set $\Lambda$ of key polynomials is said to be a \textbf{complete set of key polynomials for $\nu$} if for every $f\in K[x]$, there exists $Q\in \Lambda$ such that $\nu_Q(f)=\nu(f)$. One of the main results of this paper is the following:

\begin{Teo}\label{Theoremexistencecompleteseqkpol}
Every valuation $\nu$ on $K[x]$ admits a complete set of key polynomials.
\end{Teo}

Another way of describing extensions of valuations from $K$ to $K[x]$ is the theory of pseudo-convergent sequences developed by Kaplansky in \cite{Kap}. He uses this theory to determine whether a maximal immediate extension of the valued field $(K,\nu)$ is unique (up to isomorphism). For a valued field $(K,\nu)$, a \textbf{pseudo-convergent sequence} is a well-ordered subset $\{a_{\rho}\}_{\rho<\lambda}$ of $K$, without last element, such that
\[
\nu(a_\sigma-a_\rho)<\nu(a_\tau-a_\sigma)\mbox{ for all }\rho<\sigma<\tau<\lambda.
\]
For a given pseudo-convergent sequence $\{a_{\rho}\}_{\rho<\lambda}$ it is easy to show that either $\nu(a_\rho)<\nu(a_\sigma)$ for all $\rho<\sigma<\lambda$ or there is $\rho<\lambda$ such that $\nu(a_\sigma)=\nu(a_\rho)$ for every $\rho<\sigma<\lambda$. If we set $\gamma_\rho:=\nu(a_{\rho+1}-a_\rho)$, then $\nu(a_\sigma-a_\rho)=\gamma_\rho$ for every $\rho<\sigma<\lambda$. Hence, the sequence $\{\gamma_\rho\}_{\rho<\lambda}$ is an increasing subset of $\Gamma$. An element $a\in K$ is said to be a \textbf{limit} of the pseudo-convergent sequence $\{a_\rho\}_{\rho<\lambda}$ if $\nu(a-a_\rho)=\gamma_\rho$ for every
$\rho<\lambda$.

One can prove that for every polynomial $f(x)\in K[x]$, there exists $\rho_f<\lambda$ such that either
\begin{equation}\label{condforpscstotra}
\nu(f(a_\sigma))=\nu(f(a_{\rho_f}))\mbox{ for every }\rho_f\leq \sigma<\lambda,
\end{equation}
or
\begin{equation}\label{condforpscstoalg}
\nu(f(a_\sigma))>\nu(f(a_{\rho}))\mbox{ for every }\rho_f\leq \rho< \sigma<\lambda.
\end{equation}
If case (\ref{condforpscstotra}) happens, we say that the value of $f$ is fixed by $\{a_\rho\}_{\rho<\lambda}$ (or that $\{a_\rho\}_{\rho<\lambda}$ fixes the value of $f$). A pseudo-convergent sequence $\{a_\rho\}_{\rho<\lambda}$ is said to be of \textbf{transcendental type} if for every polynomial $f(x)\in K[x]$ the condition (\ref{condforpscstotra}) holds.
Otherwise, $\{a_\rho\}_{\rho<\lambda}$ is said to be of \textbf{algebraic type}, i.e., if there exists at least one polynomial for which condition (\ref{condforpscstoalg}) holds.

The concept of key polynomials appears in the approach to local uniformization by Spivakovsky.
On the other hand, the concept of pseudo-convergent sequence plays an important role in the work of Knaf and Kuhlmann (see \cite{KK_1}) on the local uniformization problem. In this paper, we present a comparison between the concepts of key polynomials and pseudo-convergent sequences. More specifically, we prove the following:

\begin{Teo}\label{compthemkppsc}
Let $\nu$ be a valuation on $K[x]$ and let $\{a_\rho\}_{\rho<\lambda}\subset K$ be a pseudo-convergent sequence, without a limit in $K$, for which $x$ is a limit. If $\{a_\rho\}_{\rho<\lambda}$ is of transcendental type, then $\Lambda:=\{x-a_\rho\mid \rho<\lambda\}$ is a complete set of key polynomials for $\nu$. On the other hand, if $\{a_\rho\}_{\rho<\lambda}$ is of algebraic type, then every polynomial $q(x)$ of minimal degree among the polynomials not fixed by $\{a_\rho\}_{\rho<\lambda}$ is a limit key polynomial for $\nu$.
\end{Teo}

\section{Key polynomials}
We will assume throughout this paper that $K$ is a field, $\nu$ a valuation of $K[x]$, non-trivial on $K$ with $\nu(x)\geq 0$. We begin by making some remarks.
\begin{Obs}\label{exlinearnonlinepolarepsk}
\begin{description}
\item[(i)] Every linear polynomial $x-a$ is a key polynomial (of level $\epsilon(x-a)=\nu(x-a)$).

\item[(ii)] Take a polynomial $f(x)\in K[x]$ of degree greater than one and assume that there exists $a\in K$ such that
$\nu(\partial_bf(a))=\nu(\partial_bf(x))$ for every $b\in\N$ (note that such an $a$ always exists if the assumptions of Theorem \ref{compthemkppsc} hold and the pseudo-convergent sequence is transcendental or is algebraic and $\deg(f)\leq\deg(q)$). Write
\[
f(x)=f(a)+\sum_{i=1}^n\partial_if(a)(x-a)^i
\]
and take $h\in\{1,\ldots,n\}$ such that
\[
\nu(\partial_hf(x))+h\nu(x-a)=\min_{1\leq i\leq n}\{\nu(\partial_if(x))+i\nu(x-a)\}.
\]
If $\nu(f(a))<\nu(\partial_hf(x))+h\nu(x-a)$, then $\nu(f(x))=\nu(f(a))$ and hence
\[
\frac{\nu(f(x))-\nu(\partial_if(x))}{i}<\nu(x-a)
\]
for every $i$, $1\leq i\leq n$. Consequently, $\epsilon(f)<\nu(x-a)=\epsilon(x-a)$ and hence $f$ is not a key polynomial.
On the other hand, if
$$
\nu(\partial_hf(x))+h\nu(x-a)\leq\nu(f(a)),
$$
then
\begin{equation}\label{eqtaylorexpwithpcs}
\nu(f(x))\geq\nu(\partial_hf(x))+h\nu(x-a)
\end{equation}
and if the equality holds in (\ref{eqtaylorexpwithpcs}), then
\[
\epsilon(f)=\frac{\nu(f(x))-\nu(\partial_hf(x))}{h}=\nu(x-a)=\epsilon(x-a)
\]
and hence $f$ is not a key polynomial. In other words, the only situation when $f$ may be a key polynomial is when
$$
f(x)>\min_{1\leq i\leq n}\{f(a),\nu(\partial_if(x))+i\nu(x-a)\}.
$$
\end{description}
\end{Obs}

\begin{Obs}
We observe that if $Q$ is a key polynomial of level $\epsilon:=\epsilon(Q)$, then for every polynomial $f\in K[x]$ with
$\deg(f)<\deg(Q)$ and every $b\in\N$ we have
\begin{equation}\label{eqpolyndegsmallkeypol}
\nu(\partial_b(f))>\nu(f)-b\epsilon.
\end{equation}
Indeed, from the definition of key polynomial we have that $\epsilon>\epsilon(f)$. Hence, for every $b\in\N$ we have
\[
\frac{\nu(f)-\nu(\partial_b(f))}{b}\leq\epsilon(f)<\epsilon
\]
and this implies (\ref{eqpolyndegsmallkeypol}).
\end{Obs}
Let
\[
I(f)=\left\{b\in\N\left|\frac{\nu(f)-\nu(\partial_bf)}{b}=\epsilon(f)\right.\right\}
\]
and $b(f)=\min I(f)$.
\begin{Lema}
Let $Q$ be a key polynomial and take $f,g\in K[x]$ such that
$$
\deg(f)<\deg(Q)
$$
and
$$
\deg(g)<\deg(Q).
$$
Then for $\epsilon:=\epsilon(Q)$ and any $b\in\N$ we have the following:
\begin{description}\label{lemaonkeypollder}
\item[(i)] $\nu(\partial_b(fg))>\nu(fg)-b\epsilon$
\item[(ii)] If $\nu_Q(fQ+g)<\nu(fQ+g)$ and $b\in I(Q)$, then $\nu(\partial_b(fQ+g))=\nu(fQ)-b\epsilon$;
\item[(iii)] If $h_1,\ldots,h_s$ are polynomials such that $\deg(h_i)<\deg(Q)$ for every $i=1,\ldots, s$ and
$\displaystyle\prod_{i=1}^sh_i=qQ+r$ with $\deg(r)<\deg(Q)$, then
\[
\nu(r)=\nu\left(\prod_{i=1}^sh_i\right)<\nu(qQ).
\]
\end{description}
\end{Lema}

\begin{proof}
\textbf{(i)} Since $\deg(f)<\deg(Q)$ and $\deg(g)<\deg(Q)$, for each $j\in \N$, we have
\[
\nu(\partial_jf)>\nu(f)-j\epsilon\mbox{ and }\nu(\partial_jg)>\nu(g)-j\epsilon.
\]
This, and the fact that
\[
\partial_b(fg)=\sum_{j=0}^b\partial_jf\partial_{b-j}g,
\]
imply that
\[
\nu(\partial_b(fg))\geq\min_{0\leq j\leq b}\{\nu(\partial_jf)+\nu(\partial_{b-j}g)\}>\nu(fg)-b\epsilon.
\]

\textbf{(ii)} If $\nu_Q(fQ+g)<\nu(fQ+g)$, then $\nu(fQ)=\nu(g)$. Hence, 
\[
\nu(\partial_bg)>\nu(g)-b\epsilon=\nu(fQ)-b\epsilon.
\]
Moreover, for every $j\in\N$, we have
\[
\nu(\partial_j f\partial_{b-j}Q)=\nu(\partial_jf)+\nu(\partial_{b-j}Q)>\nu(f)-j\epsilon+\nu(Q)-(b-j)\epsilon=\nu(fQ)-b\epsilon.
\]
Therefore,
\[
\nu(\partial_b(fQ+g))=\nu\left(f\partial_bQ+\sum_{j=1}^b\partial_jf\partial_{b-j}Q+\partial_bg\right)=\nu(fQ)-b\epsilon.
\]

\textbf{(iii)} We proceed by induction on $s$. If $s=1$, then $h_1=qQ+r$ with
$$
\deg(h_1)<\deg(Q),
$$
which implies that $h_1=r$ and $q=0$. Our result follows immediately.

Next, consider the case $s=2$. Take $f,g\in K[x]$ such that $\deg(f)<\deg(Q)$, $\deg(g)<\deg(Q)$ and write $fg=qQ+r$ with
$\deg(r)<\deg(Q)$. Then
$$
\deg(q)<\deg(Q)
$$
and for $b\in I(Q)$ we have
\[
\nu\left(\partial_b(qQ)\right)=\nu\left(\sum_{j=0}^b\partial_jq\partial_{b-j}Q\right)=\nu(qQ)-b\epsilon.
\]
This and part \textbf{(i)} imply that
\begin{displaymath}
\begin{array}{rcl}
\nu(qQ)-b\epsilon &=& \nu\left(\partial_b(qQ)\right)= \nu(\partial_b(fg)-\partial_b(r))\\
                  &\geq &\min\{\nu\left(\partial_b(fg)\right),\nu\left(\partial_b(r)\right)\}\\
                  &>&\min\{\nu(fg),\nu(r)\}-b\epsilon.

\end{array}
\end{displaymath}
and consequently
\begin{equation}\label{equationwithepsilon}
\nu(r)=\nu(fg)<\nu(qQ).
\end{equation}

Assume now that $s>2$ and define $\displaystyle h:=\prod_{i=1}^{s-1}h_i$. Write $h=q_1Q+r_1$ with $\deg(r_1)<\deg(Q)$. Then by the induction hypothesis we have
$$
\nu(r_1)=\nu(h)<\nu(q_1Q)
$$
and hence
\[
\nu\left(\prod_{i=1}^sh_i\right)=\nu(r_1h_s)<\nu(q_1h_sQ).
\]
Write $r_1h_s=q_2Q+r_2$. Then, by equation (\ref{equationwithepsilon}) we have
\[
\nu(r_2)=\nu(r_1h_s)<\nu(q_2Q).
\]
If $\displaystyle \prod_{i=1}^sh_i=qQ+r$ with $\deg(r)<\deg(Q)$, then
\[
qQ+r=\prod_{i=1}^sh_i=hh_s=(q_1Q+r_1)h_s=q_1h_sQ+r_1h_s=q_1h_sQ+q_2Q+r_2
\]
and hence $q=q_1h_s+q_2$ and $r=r_2$. Therefore,
\[
\nu(qQ)\geq\min\{\nu(q_1h_sQ),\nu(q_2Q)\}>\nu(r_1h_s)=\nu(r)=\nu\left(\prod_{i=1}^sh_i\right).
\]
This is what we wanted to prove.
\end{proof}

We denote by $p$ the \textbf{exponent characteristic} of $K$, that is, $p=1$ if $\Ch(K)=0$ and $p=\Ch(K)$ if $\Ch(K)>0$.
\begin{Prop}\label{propaboutpseudkeyool}
Let $Q\in K[x]$ be a key polynomial and set $\epsilon:=\epsilon(Q)$. Then the following hold:
\begin{description}
\item[(i)] Every element in $I(Q)$ is a power of $p$;
\item[(ii)] $Q$ is irreducible.
\end{description}
\end{Prop}
\begin{proof}
\textbf{(i)} Take $b\in I(Q)$ and assume, aiming for contradiction, that $b$ is not a power of $p$. Write $b=p^tr$ where $r>1$ is prime to $p$. Then, by Lemma 6 of \cite{Kap}, $\binom{b}{p^t}$ is prime to $p$ and hence $\nu\binom{b}{p^t}=0$. Since $\binom{b}{p^t}\partial_b=\partial_{p^t}\circ\partial_{b'}$ for $b'=b-p^t$, we have
\[
\nu(\partial_{b'}Q)-\nu(\partial_bQ)=\nu(\partial_{b'}Q)-\nu(\partial_{p^t}(\partial_{b'}Q))\leq p^t\epsilon(\partial_{b'}(Q))<p^t\epsilon
\]
because $\deg(\partial_{b'}Q)<\deg(Q)$ and $Q$ is a key polynomial. Hence,
\[
b\epsilon=\nu(Q)-\nu(\partial_bQ)=\nu(Q)-\nu(\partial_{b'}Q)+\nu(\partial_{b'}Q)-\nu(\partial_bQ)<
b'\epsilon+p^t\epsilon=b\epsilon,
\]
which gives the desired contradiction.

\textbf{(ii)} If $Q=gh$ for non-constant polynomials $g,h\in K[x]$, then by Lemma \ref{lemaonkeypollder} \textbf{(i)}, we would have for $b\in I(Q)$ that
\[
\nu(\partial_bQ)>\nu(Q)-b\epsilon,
\]
which is a contradiction to the definition of $b$ and $\epsilon$.
\end{proof}

We present an example to show that $\nu_q$ does not need to be a valuation for a general polynomial $q(x)\in K[x]$.
\begin{Exa}
Consider a valuation $\nu$ in $K[x]$ such that $\nu(x)=\nu(a)=1$ for some $a\in K$. Take $q(x)=x^2+1$ (which can be irreducible, for instance, if $K=\R$ or $K=\mathbb F_p$ and $-1$ is not a quadratic residue $\mod p$). Since $x^2-a^2=(x^2+1)-(a^2+1)$ we have
\[
\nu_q(x^2-a^2)=\min\{\nu(x^2+1),\nu(a^2+1)\}=0.
\]
On the other hand, $\nu_q(x+a)=\nu(x+a)\geq\min\{\nu(a),\nu(x)\}=1$ (and the same holds for $\nu_q(x-a)$). Hence
\[
\nu_q(x^2-a^2)=0<1+1\leq\nu_q(x-a)+\nu_q(x+a)
\]
which shows that $\nu_q$ is not a valuation.
\end{Exa}

If $f=f_0+f_1q+\ldots+f_nq^n$ is the $q$-standard decomposition of $f$ we set
\[
S_q(f):=\{i\in\{0,\ldots, n\}\mid \nu(f_iq^i)=\nu_q(f)\}\mbox{ and }\delta_q(f)=\max S_q(f).
\]
\begin{Prop}\label{proptruncakeypolval}
If $Q$ is a key polynomial, then $\nu_Q$ is a valuation of $K[x]$.
\end{Prop}
\begin{proof}
One can easily see that $\nu_Q(f+g)\geq\min\{\nu_Q(f),\nu_Q(g)\}$ for every $f,g\in K[x]$. It remains to prove that
$\nu_Q(fg)=\nu_Q(f)+\nu_Q(g)$ for every $f,g\in K[x]$. Assume first that $\deg(f)<\deg(Q)$ and $\deg(g)<\deg(Q)$ and let $fg=aQ+c$ be the $Q$-standard expansion of $fg$. By Lemma \ref{lemaonkeypollder} \textbf{(iii)} we have
\[
\nu(fg)=\nu(c)<\nu(aQ)
\]
and hence
\[
\nu_Q(fg)=\min\{\nu(aQ),\nu(c)\}=\nu(c)=\nu(fg)=\nu(f)+\nu(g)=\nu_Q(f)+\nu_Q(g).
\]

Now assume that $f,g\in K[x]$ are any polynomials and consider the $Q$-expansions
\[
f=f_0+\ldots+f_nQ^n\mbox{ and }g=g_0+\ldots+g_mQ^m
\]
of $f$ and $g$. Then, using the first part of the proof, we obtain
\[
\nu_Q(fg)\geq\min_{i,j}\{\nu_Q(f_ig_jQ^{i+j})\}=\min_{i,j}\{\nu_Q(f_iQ^i)+\nu_Q(g_jQ^j)\}=\nu_Q(f)+\nu_Q(g).
\]
For each $i\in\{0,\ldots,n\}$ and $j\in\{0,\ldots,m\}$, let $f_ig_j=a_{ij}Q+c_{ij}$ be the $Q$-standard expansion of $f_ig_j$. Then, by Lemma \ref{lemaonkeypollder} \textbf{(iii)}, we have
\[
\nu(f_iQ^i)+\nu(g_jQ^j)=\nu(f_ig_j)+\nu(Q^{i+j})=\nu(c_{ij})+\nu(Q^{i+j})=\nu(c_{ij}Q^{i+j}).
\]
Let
\[
i_0=\min\{i\mid\nu_Q(f)=\nu(f_iQ^i)\}\mbox{ and }j_0=\min\{j\mid\nu_Q(g)=\nu(g_jQ^j)\},
\]
and set $k_0:=i_0+j_0$. Then for every $i<i_0$ or $j<j_0$ we have
\begin{equation}\label{eqnatnaksdjs}
\min\left\{\nu(a_{ij}Q^{i+j+1}),\nu(c_{ij}Q^{i+j})\right\}=\nu(f_iQ^i)+\nu(g_jQ^j)>\nu(c_{i_0j_0}Q^{k_0}).
\end{equation}
Let $fg=a_0+a_1Q+\ldots+a_rQ^r$ be the $Q$-standard expansion of $fg$. Then
\[
a_{k_0}=\sum_{i+j+1=k_0}a_{ij}+\sum_{i+j=k_0}c_{ij}.
\]
This and equation (\ref{eqnatnaksdjs}) give us that
\[
\nu(a_{k_0}Q^{k_0})=\nu(c_{i_0j_0}Q^{k_0})=\nu(f_{i_0}Q^{i_0})+\nu(g_{j_0}Q^{j_0})=\nu_Q(f)+\nu_Q(g).
\]
Therefore,
\[
\nu_Q(fg)=\min_{0\leq k\leq r}\{\nu(a_kQ^k)\}\leq \nu_Q(f)+\nu_Q(g),
\]
which completes the proof.

\end{proof}

\begin{Prop}\label{Propdificil}
Let $Q\in K[x]$ be a key polynomial and set $\epsilon:=\epsilon(Q)$. For any $f\in K[x]$ the following hold:
\begin{description}
\item[(i)] For any $b\in\N$ we have
\begin{equation}\label{eqthatcompvalutrunc}
\frac{\nu_Q(f)-\nu_Q(\partial_bf)}{b}\leq \epsilon;
\end{equation}
\item[(ii)] If $S_Q(f)\neq\{0\}$, then the equality in (\ref{eqthatcompvalutrunc}) holds for some $b\in\N$;

\item[(iii)] If for some $b\in\N$, the equality in (\ref{eqthatcompvalutrunc}) holds and $\nu_Q(\partial_bf)=\nu(\partial_bf)$, then
$\epsilon(f)\geq\epsilon$. If in addition, $\nu(f)>\nu_Q(f)$, then $\epsilon(f)>\epsilon$.
\end{description}
\end{Prop}
Fix a key polynomial $Q$ and $h\in K[x]$ with $\deg(h)<\deg(Q)$. Then, for every $b\in\N$ the Leibnitz rule for derivation gives us that
\begin{equation}
\partial_b(hQ^n)=\sum_{b_0+\ldots+b_r=b}T_b(b_0,\ldots,b_r)
\end{equation}
where
\[
T_b(b_0,\ldots, b_r):=\partial_{b_0}h\left(\prod_{i=1}^r\partial_{b_i}Q\right)Q^{n-r}.
\]

In order to prove Proposition \ref{Propdificil}, we will need the following result:
\begin{Lema}\label{Lemamagic3}
Let $Q$ be a key polynomial, $h\in K[x]$ with $\deg(h)<\deg(Q)$ and set $\epsilon:=\epsilon(Q)$. For any $b\in\N$ we have
\[
\nu_Q(T_b(b_0,\ldots,b_r))\geq \nu(hQ^n)-b\epsilon.
\]
Moreover, if either $b_0>0$ or $b_i\notin I(Q)$ for some $i=1,\ldots, r$, then
\[
\nu_Q(T_b(b_0,\ldots,b_r))> \nu(hQ^n)-b\epsilon.
\]
\end{Lema}
\begin{proof}
Since $\deg(h)<\deg(Q)$ and $Q$ is a key polynomial we have $\epsilon(h)<\epsilon$. Hence, if $b_0>0$ we have
\[
\nu(\partial_{b_0}h)\geq \nu(h)-b_0\epsilon(h)>\nu(h)-b_0\epsilon.
\]
On the other hand, for every $i=1,\ldots, r$, by definition of $\epsilon$ we have
\[
\nu(\partial_{b_i}Q)\geq \nu(Q)-b_i\epsilon,
\]
and if $b_i\notin I(Q)$ we have
\[
\nu(\partial_{b_i}Q)> \nu(Q)-b_i\epsilon.
\]
Since $\nu_Q(\partial_{b_0}h)=\nu(\partial_{b_0}h)$ and $\nu_Q(\partial_{b_i}Q)=\nu(\partial_{b_i}Q)$, we have
\begin{displaymath}
\begin{array}{rcl}
\nu_Q(T_b(b_0,\ldots,b_r))&=&\displaystyle\nu_Q\left(\partial_{b_0}h\left(\prod_{i=1}^r\partial_{b_i}Q\right)Q^{n-r}\right)\\
                        &=&\displaystyle\nu_Q(\partial_{b_0}h)+\sum_{i=1}^r\nu_Q(\partial_{b_i}Q)+(n-r)\nu_Q(Q)\\
                        &\geq &\displaystyle\nu(h)-b_0\epsilon+\sum_{i=1}^r\left(\nu(Q)-b_i\epsilon\right)+(n-r)\nu(Q)\\
                        &\geq &\nu(hQ^n)-b\epsilon.
\end{array}
\end{displaymath}
Moreover, if $b_0>0$ or $b_i\notin I(Q)$ for some $i=1,\ldots, r$, then the inequality above is strict.
\end{proof}

\begin{Cor}\label{Coroaboutderib}
For every $b\in\N$ we have $\nu_Q\left(\partial_b(aQ^n)\right)\geq\nu(aQ^n)-b\epsilon$.
\end{Cor}

\begin{proof}[Proof of Proposition \ref{Propdificil}]
\textbf{(i)} Take any $f\in K[x]$ and consider its $Q$-standard expansion $f=f_0+f_1Q+\ldots+f_nQ^n$. For each $i=0,\ldots,n$, Corollary \ref{Coroaboutderib} gives us that
\[
\nu_Q\left(\partial_b(f_iQ^i)\right)\geq \nu (f_iQ^i)-b\epsilon.
\]
Hence,
\[
\nu_Q\left(\partial_b(f)\right)\geq\min_{0\leq i\leq n}\{\nu_Q(f_iQ^i)\}\geq\min_{0\leq i\leq n}\{\nu(f_iQ^i)-b\epsilon\}=\nu_Q(f)-b\epsilon.
\]
\textbf{(ii)} Assume that $S_Q(f)\neq \{0\}$ and set $j_0=\min S_Q(f)$. Then $j_0=p^er$ for some $e\in\N\cup\{0\}$  and some $r\in\N$ with $(r,p)=1$. We set $b:=p^eb(Q)$ and will prove that $\nu_Q(\partial_b(f))=\nu_Q(f)-b\epsilon$.

Write
\[
f_{j_0}\left(\partial_{b(Q)}Q\right)^{p^e}=rQ+h
\]
for some $r,h\in K[x]$ and $\deg(h)<\deg(Q)$ (note that $h\ne0$ because $Q$ is irreducible and $Q\nmid f_{j_0}$ and $Q\nmid \partial_{b(Q)}Q$). Then Lemma \ref{lemaonkeypollder} \textbf{(iii)} gives us that
\[
\nu(h)=\nu\left(f_{j_0}(\partial_{b(Q)}Q)^{p^e}\right).
\] 
This implies that
\begin{equation}\label{equationboa}
\nu\left(hQ^{j_0-p^e}\right)=\nu_Q(f)-b\epsilon.
\end{equation}
Indeed, we have
\begin{displaymath}
\begin{array}{rcl}
\nu\left(hQ^{j_0-p^e}\right)&=& \nu(h)+\nu\left(Q^{j_0-p^e}\right)=\nu\left(f_{j_0}(\partial_{b(Q)}Q)^{p^e}\right)+
\nu\left(Q^{j_0-p^e}\right)\\
                            &=&\nu(f_{j_0})+p^e\nu\left(\partial_{b(Q)}Q\right)+(j_0-p^e)\nu(Q)\\
                            &=&\nu(f_{j_0})+p^e\left(\nu\left(Q\right)-b(Q)\epsilon\right)+(j_0-p^e)\nu(Q)\\
                            &=&\nu(f_{j_0})+j_0\nu(Q)-p^eb(Q)\epsilon\\
                            &=&\nu(f_{j_0}Q^{j_0})-p^eb(Q)\epsilon=\nu_Q(f)-b\epsilon.\\
                            
\end{array}
\end{displaymath}

Since $f=f_0+f_1Q+\ldots+f_nQ^n$, we have $\partial_b(f)=\partial_b(f_0)+\partial_b(f_1Q)\ldots+\partial_b(f_nQ^n)$. For each $j=0,\ldots, n$, if $j\notin S_Q(f)$, then
\[
\nu_Q\left(\partial_b(f_jQ^j)\right)\geq \nu_Q(f_jQ^j)-b\epsilon>\nu_Q(f)-b\epsilon.
\]
We set
\[
h_1=\sum_{j\notin S_Q(f)}f_jQ^j.
\]
Then $\nu_Q(h_1)>\nu_Q(f)-b\epsilon$.

For each $j\in S_Q(f)$ the term $\partial_b(f_jQ^j)$ can be written as a sum of terms of the form $T_b(b_0,\ldots,b_r)$. For each $T_b(b_0,\ldots,b_r)$ we have the following cases:

\textbf{Case 1:} $b_0>0$ or $b_i\notin I(Q)$ for some $i$.\\
In this case, by Lemma \ref{Lemamagic3} we have $\nu_Q(T_b(b_0,\ldots,b_r))>\nu_Q(f)-b\epsilon$. In particular, if $h_2$ is the sum of all these terms, then $\nu_Q(h_2)>\nu_Q(f)-b\epsilon$.

\textbf{Case 2:} $b_0=0$ and $b_i\in I(Q)$ for every $i=1,\ldots,r$ but $b_{i_0}\neq b(Q)$ for some $i_0=1,\ldots,r$.\\
This implies, in particular, that $j\geq j_0$ and since $b=p^eb(Q)$ we must have $r<p^e$. Hence
\[
T_b(b_0,b_1,\ldots,b_r)=\partial_{b_0}f_j\left(\prod_{i=1}^r\partial_{b_i}Q\right)Q^{j-r}=sQ^{j_0-p^e+1}
\]
for some $s\in K[x]$.

\textbf{Case 3:} $b_0=0$, $j>j_0$ and $b_i=b(Q)$ for every $i=1,\ldots,r$.\\
Since $b=p^eb(Q)$, $b_i=b(Q)$ and $\displaystyle\sum_{i=1}^rb_i=b$ we must have $r=p^e$. Hence
\[
T_b(b_0,b_1,\ldots,b_r)=f_j\left(\partial_{b(Q)}Q\right)^{p^e}Q^{j-p^e}=s'Q^{j_0-p^e+1}
\]
for some $s'\in K[x]$.\\

\textbf{Case 4:} $b_0=0$, $j=j_0$ and $b_i=b(Q)$ for every $i=1,\ldots,r$.\\
In this case we have
\begin{equation}\label{caseintport}
\begin{array}{rcl}
T_b(b_0,b_1,\ldots,b_r)&=&f_{j_0}\left(\partial_{b(Q)}Q\right)^{p^e}Q^{j_0-p^e}\\
                       &=&\left(h-rQ\right)Q^{j_0-p^e}\\
                       &=&hQ^{j_0-p^e}-rQ^{j_0-p^e+1}.
\end{array}
\end{equation}

Observe that the number of times that the term (\ref{caseintport}) appears in $\partial_b(f_{j_0}Q^{j_0})$ is $\binom{j_0}{p^e}$, that is, the number of ways that one can choose a subset with $p^e$ elements in a set of $j_0$ elements.

Therefore, we can write
\[
\partial_b(f)=\binom{j_0}{p^e}hQ^{j_0-p^e}+\left(s+s'-\binom{j_0}{p^e}r\right)Q^{j_0-p^e+1}+h_1+h_2
\]
Since $p\nmid \binom{j_0}{p^e}$ the equation (\ref{equationboa}) gives us that
\[
\nu\left(\binom{j_0}{p^e}hQ^{j_0-p^e}\right)=\nu_Q(f)-b\epsilon.
\]
Then
\[
\nu_Q\left(\binom{j_0}{p^e}hQ^{j_0-p^e}+\left(s+s'-\binom{j_0}{p^e}r\right)Q^{j_0-p^e+1}\right)\leq \nu_Q(f)-b\epsilon.
\]
This and the fact that $\nu_Q(h_1+h_2)>\nu_Q(f)-b\epsilon$ imply that $\nu_Q\left(\partial_b(f)\right)\leq\nu_Q(f)-b\epsilon$. This concludes the proof of \textbf{(ii)}.

\textbf{(iii)} The assumptions on $b$ give us
\[
\frac{\nu_Q(f)-\nu_Q(\partial_bf)}{b}= \epsilon
\]
and
\[
\nu_Q(\partial_bf)=\nu(\partial_bf).
\]
Consequently,
\[
\epsilon(f)\geq \frac{\nu(f)-\nu(\partial_bf)}{b}\geq\frac{\nu_Q(f)-\nu_Q(\partial_bf)}{b}= \epsilon.
\]
In the inequality above, one can see that if $\nu(f)>\nu_Q(f)$, then $\epsilon(f)>\epsilon$.

\end{proof}

\begin{Prop}\label{Propcompkeypol}
For two key polynomials $Q,Q'\in K[x]$ we have the following:
\begin{description}
\item[(i)] If $\deg(Q)<\deg(Q')$, then $\epsilon(Q)<\epsilon(Q')$;
\item[(ii)] If $\epsilon(Q)<\epsilon(Q')$, then $\nu_Q(Q')<\nu(Q')$;
\item[(iii)] If $\deg(Q)=\deg(Q')$, then
\begin{equation}\label{eqwhdegsame}
\nu(Q)<\nu(Q')\Llr \nu_Q(Q')<\nu(Q')\Llr \epsilon(Q)<\epsilon(Q').
\end{equation}
\end{description}
\end{Prop}
\begin{proof}
Item \textbf{(i)} follows immediately from the the definition of key polynomial (in fact, the same holds if we substitute $Q$ for any $f\in K[x]$).

In order to prove \textbf{(ii)} we set $\epsilon:=\epsilon(Q)$ and $b':=b(Q')$. By \textbf{(i)} of Proposition \ref{Propdificil}, we have
\[
\nu_Q(Q')\leq \nu_Q(\partial_{b'}Q')+b'\epsilon.
\]
Since $\epsilon(Q)<\epsilon(Q')$, we also have
\[
\nu(\partial_{b'}Q')+b'\epsilon< \nu(\partial_{b'}Q')+b'\epsilon(Q')=\nu(Q').
\]
This, and the fact that $\nu_Q(\partial_{b'}Q')\leq \nu(\partial_{b'}Q')$, imply that $\nu_Q(Q')<\nu(Q')$.

Now assume that $\deg(Q)=\deg(Q')$ and let us prove (\ref{eqwhdegsame}). Since
$$
\deg(Q)=\deg(Q')
$$
and both $Q$ and $Q'$ are monic, the $Q$-standard expansion of $Q'$ is given by
$$
Q'=Q+(Q-Q').
$$
Hence
\[
\nu_Q(Q')=\min\{\nu(Q),\nu(Q-Q')\}.
\]
The first equivalence follows immediately from this. In view of part \textbf{(ii)}, it remains to prove that if $\nu_Q(Q')<\nu(Q')$, then $\epsilon(Q)<\epsilon(Q')$. Since $\nu_Q(Q')<\nu(Q')$ we have $S_Q(Q')\neq \{0\}$. Hence, by Proposition \ref{Propdificil} \textbf{(ii)}, the equality holds in (\ref{eqthatcompvalutrunc}) (for $f=Q'$) for some $b\in\N$. Moreover, since $\deg(Q)=\deg(Q')$, we have
$\deg(\partial_bQ')<\deg(Q)$ and consequently $\nu_Q(\partial_bQ')=\nu(\partial_bQ')$. Then Proposition \ref{Propdificil} \textbf{(iii)} implies that $\epsilon(Q)<\epsilon(Q')$.
\end{proof}

For a key polynomial $Q\in K[x]$, let
\[
\alpha(Q):=\min\{\deg(f)\mid \nu_Q(f)< \nu(f)\}
\]
(if $\nu_Q=\nu$, then set $\alpha(Q)=\infty$) and
\[
\Psi(Q):=\{f\in K[x]\mid f\mbox{ is monic},\nu_Q(f)< \nu(f)\mbox{ and }\alpha(Q)=\deg (f)\}.
\]
\begin{Lema}\label{lemmapsikeypoly}
If $Q$ is a key polynomial, then every element $Q'\in\Psi(Q)$ is also a key polynomial. Moreover, $\epsilon(Q)<\epsilon(Q')$.
\end{Lema}
\begin{proof}
By assumption, we have $\nu_Q(Q')<\nu(Q')$, hence $S_{Q}(Q')\neq \{0\}$. This implies, by Proposition \ref{Propdificil} \textbf{(ii)}, that there exists $b\in\N$ such that
\[
\nu_Q(Q')-\nu_Q(\partial_b Q')=b\epsilon(Q).
\]
Since $\deg(\partial_b Q')<\deg(Q')=\alpha(Q)$, we have $\nu_Q(\partial_b Q')=\nu(\partial_b Q')$. Consequently, by Proposition \ref{Propdificil} \textbf{(iii)}, $\epsilon(Q)<\epsilon(Q')$.

Now take any polynomial $f\in K[x]$ such that $\deg(f)<\deg(Q')=\alpha(Q)$. In particular, $\nu_Q(f)=\nu(f)$. Moreover, for every
$b\in\N$, $\deg(\partial_b f)<\deg(Q')= \alpha(Q)$ which implies that $\nu_Q(\partial_bf)=\nu(\partial_b f)$. Then, for every $b\in\N$,
\[
\frac{\nu(f)-\nu(\partial_bf)}{b}=\frac{\nu_Q(f)-\nu_Q(\partial_bf)}{b}\leq \epsilon(Q)<\epsilon(Q').
\]
This implies that $\epsilon(f)<\epsilon(Q')$, which shows that $Q'$ is a key polynomial.

\end{proof}
\begin{Teo}\label{definofkeypol}
A polynomial $Q$ is a key polynomial if and only if there exists a key polynomial $Q_-\in K[x]$ such that $Q\in \Psi(Q_-)$ or the following conditions hold:
\begin{description}
\item[(K1)] $\alpha(Q_-)=\deg (Q_-)$
\item[(K2)] the set $\{\nu(Q')\mid Q'\in\Psi(Q_-)\}$ does not contain a maximal element
\item[(K3)] $\nu_{Q'}(Q)<\nu(Q)$ for every $Q'\in \Psi(Q_-)$
\item[(K4)] $Q$ has the smallest degree among polynomials satisfying \textbf{(K3)}.
\end{description}
\end{Teo}

\begin{proof}
We will prove first that if such $Q_-$ exists, then $Q$ is a key polynomial. The case when $Q\in \Psi(Q_-)$ follows from Lemma \ref{lemmapsikeypoly}. Assume now that \textbf{(K1) - (K4)} hold. Take $f\in K[x]$ such that $\deg(f)<\deg(Q)$. This implies that $\deg(\partial_bQ)<\deg(Q)$ and $\deg(\partial_bf)<\deg(Q)$ for every $b\in\N$. Hence, by \textbf{(K4)}, there exists $Q'\in\Psi(Q_-)$ such that
\[
\nu_{Q'}(f)=\nu(f), \nu_{Q'}(\partial_bf)=\nu(\partial_bf) \mbox{ and }\nu_{Q'}(\partial_bQ)=\nu(\partial_bQ)\mbox{ for every }b\in\N.
\]
We claim that $\epsilon(Q')<\epsilon(Q)$. If not, by Proposition \ref{Propcompkeypol} \textbf{(i)}, we would have
$\deg(Q)\leq\deg(Q')$. Since $\nu_{Q'}(Q)<\nu(Q)$, this implies that $\deg(Q)=\deg(Q')$. This and Proposition \ref{Propcompkeypol} \textbf{(iii)} give us that $\epsilon(Q')<\epsilon(Q)$ which is a contradiction.

Now,
\[
\epsilon(f)\leq \frac{\nu(f)-\nu(\partial_bf)}{b}=\frac{\nu_{Q'}(f)-\nu_{Q'}(\partial_bf)}{b}\leq\epsilon(Q')<\epsilon(Q).
\]
Hence $Q$ is a key polynomial.

For the converse, take a key polynomial $Q\in K[x]$ and consider the set
\[
\mathcal S:=\{Q'\in K[x]\mid Q'\mbox{ is a key polynomial and }\nu_{Q'}(Q)<\nu(Q)\}.
\]
Observe that $\mathcal S\neq\emptyset$. Indeed, if $\deg(Q)>1$, then every key polynomial $x-a\in\mathcal S$. If $Q=x-a$, then there exits $b\in K$ such that $\nu(b)<\min\{\nu(a),\nu(x)\}$. Therefore, $x-b\in \mathcal S$.

If there exists a key polynomial $Q_-\in \mathcal S$ such that $\deg(Q)=\deg(Q_-)$ , then we have $Q\in \Psi(Q_-)$ and we are done. Hence, assume that every polynomial $Q'\in \mathcal S$ has degree smaller that $\deg(Q)$. 

Assume that there exists $Q_-\in \mathcal S$ such that for every $Q'\in \mathcal S$ we have
\begin{equation}\label{eqmaxphi}
(\deg(Q_-),\nu(Q_-))\geq ((\deg(Q'),\nu(Q'))
\end{equation}
in the lexicographical ordering. We claim that $Q\in \Psi(Q_-)$. If not, there would exist a key polynomial $Q''$ such that $\nu_{Q_-}(Q'')<\nu(Q'')$ and $\deg(Q'')<\deg(Q)$. Since $\deg(Q'')<\deg(Q)$ Proposition \ref{Propcompkeypol} \textbf{(i)} and \textbf{(ii)} give us that $\nu_{Q''}(Q)<\nu(Q)$. Hence $Q''\in \mathcal S$. The inequality (\ref{eqmaxphi}) gives us that $\deg(Q'')\leq\deg(Q_-)$. On the other hand, since $\nu_{Q_-}(Q'')<\nu(Q'')$ we must have $\deg(Q_-)=\deg(Q'')$. Hence, Proposition \ref{Propcompkeypol} \textbf{(iii)} gives us that $\nu(Q_-)<\nu(Q'')$ and this is a contradiction to the inequality (\ref{eqmaxphi}).

Now assume that for every $Q'\in \mathcal S$, there exists $Q''\in \mathcal S$ such that
\begin{equation}\label{eqmaxphihas}
(\deg(Q'),\nu(Q'))<(\deg(Q''),\nu(Q''))
\end{equation}
in the lexicographical ordering. Take $Q_-\in\mathcal S$ such that $\deg(Q_-)\geq \deg(Q')$ for every $Q'\in\mathcal S$. We will show that the conditions \textbf{(K1) - (K4)} are satisfied.  By (\ref{eqmaxphihas}), there exists $Q''\in \mathcal S$ such that
\begin{equation}\label{eqbanolimit}
(\deg(Q_-),\nu(Q_-))<(\deg(Q''),\nu(Q'')).
\end{equation}
In particular, $\deg(Q_-)=\deg(Q'')$ and $\nu(Q_-)<\nu(Q'')$. Proposition \ref{Propcompkeypol} \textbf{(iii)} gives us that
$\nu_{Q_-}(Q'')<\nu(Q'')$. Hence $\alpha(Q_-)=\deg(Q_-)$ and we have proved \textbf{(K1)}. If $Q'\in\Psi(Q_-)$, then
$\deg(Q')=\deg(Q_-)<\deg(Q)$ and hence $\nu_{Q'}(Q)<\nu(Q)$. This implies that $Q'\in\mathcal S$. The equation (\ref{eqbanolimit}) tells us that $\{\nu(Q')\mid Q'\in\Psi(Q_-)\}$ has no maximum, so we have proved \textbf{(K2)}. Now take any element $Q'\in\Psi(Q_-)$. Then $\deg(Q')<\deg(Q)$ and Proposition \ref{Propcompkeypol} \textbf{(i)} and \textbf{(ii)} give us that $\nu_{Q'}(Q)<\nu(Q)$. This proves \textbf{(K3)}. Take a polynomial $\widetilde{Q}$ with $\nu_{Q'}(\widetilde{Q})<\nu(\widetilde{Q})$ for every $Q'\in\Psi(Q_-)$ with minimal degree possible. We want to prove that $\deg(\widetilde Q)=\deg(Q)$. Assume, aiming for a contradiction, that
$\deg(\widetilde{Q})<\deg(Q)$. The first part of the proof gives us that $\widetilde Q$ is a key polynomial. Fix $Q'\in\Psi(Q_-)$. Then
$\nu_{Q'}(\widetilde Q)<\nu(\widetilde Q)$ and consequently $\deg(\widetilde{Q})=\deg(Q')=\deg(Q_-)$. Therefore
$\nu(Q')<\nu(\widetilde Q)$ for every $Q'\in \Psi(Q_-)$, which is a contradiction to (\ref{eqmaxphihas}). This concludes our proof.
\end{proof}

\begin{Def}
When conditions \textbf{(K1) - (K4)} of Theorem \ref{definofkeypol} are satisfied, we say that $Q$ is a \textbf{limit key polynomial}.
\end{Def}

\begin{Obs}
Observe that as a consequence of the proof we obtain that
$$
\epsilon(Q_-)<\epsilon(Q).
$$
\end{Obs}

\begin{proof}[Proof of Theorem \ref{Theoremexistencecompleteseqkpol}]
Consider the set
\[
\Gamma_0:=\{\nu(x-a)\mid a\in K\}.
\]
We have two possibilities:

\begin{itemize}
\item $\Gamma_0$ has a maximal element
\end{itemize}
Set $Q_0:=x-a_0$ where $a_0\in K$ is such that $\nu(x-a_0)$ is a maximum of $\Gamma_0$. If $\nu=\nu_{Q_0}$ we are done, so assume that $\nu\neq\nu_{Q_0}$. If the set
\[
\{\nu(Q)\mid Q\in \Psi(Q_0)\}
\]
has a maximum, choose $Q_1\in \Psi(Q_0)$ such that $\nu(Q_1)$ is this maximum. If not, choose $Q_1$ as any polynomial in $\Psi(Q_0)$. Set $\Lambda_1:=\{Q_0,Q_1\}$ (ordered by $Q_0<Q_1$).

\begin{itemize}
\item $\Gamma_0$ does not have a maximal element
\end{itemize}
For every $\gamma\in \Gamma_0$ set $Q_\gamma:=x-a_\gamma$ for some $a_\gamma\in K$ such that $\nu(x-a_\gamma)=\gamma$. If for every $f\in K[x]$, there exists $\gamma\in \Gamma_0$ such that $\nu(f)=\nu_{Q_\gamma}(f)$ we are done. If not, let $Q$ be a polynomial of minimal degree among all the polynomials for which $\nu_{Q_\gamma}(Q)<\nu(Q)$ for every $\gamma\in \Gamma_0$. If $\alpha(Q)=\deg(Q)$ and the set $\{\nu(Q')\mid Q'\in \Psi(Q)\}$ contains a maximal element, choose $Q_1\in \Psi(Q)$ such that $\nu(Q_1)\geq \nu(Q')$ for every $Q'\in \Psi(Q)$. If not, set $Q_1:=Q$. Set $\Lambda_1:=\{Q_\gamma\mid \gamma\in \Gamma_0\}\cup\{Q_1\}$ (ordered by $Q_1>Q_\gamma$ for every $\gamma\in \Gamma$ and $Q_\gamma>Q_{\gamma'}$ if $\gamma>\gamma'$).

Observe that in either case, $\deg(Q_1)>\deg(Q_0)$ and for $Q,Q'\in \Lambda_1$, $Q<Q'$ if and only if $\epsilon(Q)<\epsilon(Q')$. Moreover, if $\alpha(Q_1)=\deg(Q_1)$, then $\{\nu(Q)\mid Q\in\Psi(Q_1)\}$ does not have a maximum.

Assume that for some $i\in\N$, there exists a totally ordered set $\Lambda_i$ consisting of key polynomials with the following properties:
\begin{description}
\item[(i)] there exist $Q_0,Q_1,\ldots,Q_i\in \Lambda_i$ such that $Q_i$ is the last element of $\Lambda_i$ and $\deg(Q_0)<\deg(Q_1)<\ldots<\deg(Q_i)$.
\item[(ii)] if $\alpha(Q_i)=\deg(Q_i)$, then $\Gamma_i:=\{\nu(Q)\mid Q\in \Psi(Q_i)\}$ does not have a maximum.
\item[(iii)] for $Q,Q'\in \Lambda_i$, $Q<Q'$ if and only if $\epsilon(Q)<\epsilon(Q')$.
\end{description}
If $\nu_{Q_i}\neq \nu$, then we will construct a set $\Lambda_{i+1}$ of key polynomials having the same properties (changing $i$ by
$i+1$).

Since $\nu_{Q_i}\neq \nu$, the set $\Psi(Q_i)$ is not empty. We have two cases:
\begin{itemize}
\item $\alpha(Q_i)>\deg(Q_i)$.
\end{itemize}
If $\Gamma_i$ has a maximum, take $Q_{i+1}\in\Psi(Q_i)$ such that $\nu(Q_{i+1})\geq \Gamma_i$. Otherwise, choose $Q_{i+1}$ to be any element of $\Psi(Q_i)$. Observe that if $\alpha(Q_{i+1})=\deg(Q_{i+1})$, then $\Gamma_{i+1}$ does not have a maximum. Set $\Lambda_{i+1}=\Lambda_i\cup\{Q_{i+1}\}$ with the extension of the order in $\Lambda_i$ obtained by setting $Q_{i+1}>Q$ for every $Q\in \Lambda_i$.
\begin{itemize}
\item $\alpha(Q_i)=\deg(Q_i)$.
\end{itemize}
By assumption, the set $\Gamma_i$ does not have a maximum. For each $\gamma\in\Gamma_i$, choose a polynomial
$Q_\gamma\in \Psi(Q_i)$ such that $\nu(Q_\gamma)=\gamma$. If for every $f\in K[x]$, there exists $\gamma\in \Gamma_i$ such that
$\nu_{Q_\gamma}(f)=\nu(f)$, then we are done. Otherwise, choose a monic polynomial $Q$, of smallest degree possible, such that
$\nu_{Q'}(Q)<\nu(Q)$ for every $Q'\in\Psi(Q_i)$. If $\alpha(Q)=\deg(Q)$ and $\{\nu(Q')\mid Q'\in\Psi(Q)\}$ has a maximum, we choose $Q_{i+1}$ such that $\nu(Q_{i+1})\geq\{\nu(Q')\mid Q'\in\Psi(Q)\}$. Otherwise we set $Q_{i+1}=Q$. Then set
\[
\Lambda_{i+1}:=\Lambda_i\cup\{Q_\gamma\mid \gamma\in\Gamma_i\}\cup\{Q_{i+1}\},
\]
with the extension of the order of $\Lambda_i$ given by
$$
Q_{i+1}>Q'\mbox{ for every }Q'\in \Lambda_{i+1}\setminus\{Q_{i+1}\},
$$
$Q_\gamma> Q'$ for every $\gamma\in \Gamma_i$ and $Q'\in\Lambda_i$ and $Q_\gamma>Q_{\gamma'}$ for $\gamma,\gamma'\in \Gamma_i$ with $\gamma>\gamma'$.

In all cases, the set $\Lambda_{i+1}$ has the properties \textbf{(i)}, \textbf{(ii)} and \textbf{(iii)}.

Assume now that for every $i\in\N$ the sets $\Lambda_i$ and $\Lambda_{i+1}$ can be constructed. Then we can construct a set 
\[
\Lambda_\infty:=\bigcup_{i=1}^\infty\Lambda_i
\]
of key polynomials having the property that for $Q,Q'\in \Lambda_\infty$, $Q<Q'$ if and only if $\epsilon(Q)<\epsilon(Q')$ and there are polynomials $Q_0,\ldots,Q_i,\ldots\in \Lambda_\infty$ such that
$$
\deg(Q_{i+1})>\deg(Q_i)
$$
for every $i\in\N$. This means that for every $f\in K[x]$ there exists $i\in\N$ such that $\deg(f)<\deg(Q_i)$, which implies that
$\nu_{Q_i}(f)=\nu(f)$. Therefore, $\Lambda_\infty$ is a complete set of key polynomials for $\nu$.
\end{proof}

Observe that at each stage, the same construction would work if we replaced $\Gamma_i$ by any cofinal set $\Gamma_i'$ of
$\Gamma_i$. Hence, if the rank of $\nu$ is equal to 1, then we can choose $\Gamma_i'$ to have order type at most $\omega$. Then, from the construction of the sets $\Lambda_i$ and $\Lambda_\infty$, we can conclude the following:
\begin{Cor}
If the rank of $\nu$ is equal to one, then there exists a complete sequence of key polynomials of $\nu$ with order type at most $\omega\times\omega$.
\end{Cor}

\section{Pseudo-convergent sequences}
The next two theorems justify the definitions of algebraic and transcendental pseudo-convergent sequences.
\begin{Teo}[Theorem 2 of \cite{Kap}]
If $\{a_\rho\}_{\rho<\lambda}$ is a pseudo-convergent sequence of transcendental type, without a limit in $K$, then there exists an immediate transcendental extension $K(z)$ of $K$ defined by setting $\nu(f(z))$ to be the value $\nu(f(a_{\rho_f}))$ as in condition (\ref{condforpscstotra}). Moreover, for every valuation $\mu$ in some extension $K(u)$ of $K$, if $u$ is a pseudo-limit of $\{a_\rho\}_{\rho<\lambda}$, then there exists a value preserving $K$-isomorphism from $K(u)$ to $K(z)$ taking $u$ to $z$.
\end{Teo}

\begin{Teo}[Theorem 3 of \cite{Kap}]\label{thmonalgimmext}
Let $\{a_\rho\}_{\rho<\lambda}$ be a pseudo-convergent sequence of algebraic type, without a limit in $K$, $q(x)$ a polynomial of smallest degree for which (\ref{condforpscstoalg}) holds and $z$ a root of $q(x)$. Then there exists an immediate algebraic extension of $K$ to $K(z)$ defined as follows: for every polynomial $f(x)\in K[x]$, with $\deg f<\deg q$ we set $\nu(f(z))$ to be the value $\nu(f(a_{\rho_f}))$ as in condition (\ref{condforpscstotra}). Moreover, if $u$ is a root of $q(x)$ and $\mu$ is some extension $K(u)$ of $K$ making $u$ a pseudo-limit of $\{a_\rho\}_{\rho<\lambda}$, then there exists a value preserving $K$-isomorphism from $K(u)$ to $K(z)$ taking $u$ to $z$.
\end{Teo}

For the rest of this paper, let $\{a_\rho\}_{\rho<\lambda}$ be a pseudo-convergent sequence for the valued field $(K,\nu)$, without a limit in $K$. For each $\rho<\lambda$, we denote $\nu_\rho=\nu_{x-a_\rho}$. For a polynomial $f(x)\in K[x]$ and $a\in K$ we consider the Taylor expansion of $f$ at $a$ given by
\[
f(x)=f(a)+\partial_1f(a)(x-a)+\ldots+\partial_nf(a)(x-a)^n.
\]
Assume that $\{a_\rho\}_{\rho<\lambda}$ fixes the value of the polynomials $\partial_if(x)$ for every $1\leq i\leq n$. We denote by $\beta_i$ this fixed value.
\begin{Lema}[Lemma 8 of \cite{Kap}]\label{lemmakaplvalpol}
There is an integer $h$, which is a power of $p$, such that for sufficiently large $\rho$
\[
\beta_i+i\gamma_\rho>\beta_h+h\gamma_\rho\mbox{ whenever }i\ne h\mbox{ and } \nu(f(a_\rho))=\beta_h+h\gamma_\rho.
\]
\end{Lema}

\begin{Cor}\label{correlanurhowithnu}
If $\{a_\rho\}_{\rho<\lambda}$ fixes the value of $f(x)$, then $\nu_\rho(f(x))=\nu(f(x))$. On the other hand, if $\{a_\rho\}_{\rho<\lambda}$ does not fix the value of $f(x)$, then $\nu_\rho(f(x))<\nu(f(x))$ for every $\rho<\lambda$.
\end{Cor}
\begin{proof}
By definition of $\nu_\rho$ we have
\[
\nu_\rho(f(x))=\min_{0\leq i\leq n}\{\nu(\partial_if(a_\rho)(x-a_\rho)^i)\}=\min_{0\leq i\leq n}\{\beta_i+i\gamma_\rho\},
\]
where $\beta_0:=\nu(f(a_\rho))$. This implies, using the lemma above, that
\[
\nu_\rho(f(x))=\nu(f(a_\rho)).
\]
If $\{a_\rho\}_{\rho<\lambda}$ fixes the value of $f(x)$, then $\nu(f(a_\rho))=\nu(f(x))$ for $\rho$ sufficiently large. Thus $\nu_\rho(f(x))=\nu(f(x))$. On the other hand, if $\{a_\rho\}_{\rho<\lambda}$ does not fix the value of $f(x)$, then $\nu(f(x))>\nu(f(a_\rho))=\nu_\rho(f(x))$ for every $\rho<\lambda$.
\end{proof}

\begin{proof}[Proof of Theorem \ref{compthemkppsc}]
If $\{a_\rho\}_{\rho<\lambda}$ is of transcendental type it fixes, for any polynomial $f(x)\in K[x]$, the values of the polynomials
$\partial_if(x)$ for every $0\leq i\leq n$ (here $\partial_0f:=f$). Hence, Corollary \ref{correlanurhowithnu} implies that
$\nu_\rho(f(x))=\nu(f(x))$ for sufficiently large $\rho<\lambda$, which is what we wanted to prove.

Now assume that $\{a_\rho\}_{\rho<\lambda}$ is of algebraic type. Take $\rho<\lambda$ such that
$$
\nu(q(a_\tau))>\nu(q(a_\sigma))
$$
for every $\rho<\sigma<\tau<\lambda$ and set $Q_-=x-a_\rho$. Then
\[
\nu_{Q_-}(x-a_\sigma)=\nu_{Q_-}(x-a_\rho+a_\rho-a_\sigma)=\nu(x-a_\rho)<\nu(x-a_\sigma)
\]
for every $\rho<\sigma<\lambda$. This implies that $\alpha(Q-)=1$ and then $\alpha(Q_-)=\deg (Q_-)$. Consequently, \textbf{(K1)} is satisfied. Moreover,
\[
\Psi(Q_-)=\{x-a\mid \nu_{Q_-}(x-a)<\nu(x-a)\}.
\]
In order to prove \textbf{(K2)} assume, aiming for a contradiction, that $\nu(\Psi(Q_-))$ has a maximum, let us say $\nu(x-a)$. Then, in particular, $\nu(x-a)>\nu(x-a_\sigma)$ for every $\rho<\sigma<\lambda$. This implies that $a\in K$ is a limit of  $\{a_\rho\}_{\rho<\lambda}$, which is a contradiction. Condition \textbf{(K3)} and \textbf{(K4)} follow immediately from Corollary \ref{correlanurhowithnu} and the fact that $\{\nu(x-a_\rho)\mid \rho<\lambda\}$ is cofinal in $\nu(\Psi(Q_-))$.
\end{proof}

\end{document}